\documentclass[reqno,12pt]{amsart}

\usepackage{amssymb}
\usepackage{amsmath}
\usepackage{amsthm}
\usepackage{hyperref}

\hypersetup{
pdftitle={Research Paper},
pdfauthor={Ali Akbar Yazdan Pour},
pdfsubject={Valabrega-Valla Number},
pdfcreator={Ali Akbar Yazdan Pour},
colorlinks,
linkcolor=blue,
citecolor=magenta,
anchorcolor=red,
bookmarksopen,
urlcolor=red,
filecolor=red,
}

\theoremstyle{plain}
\newtheorem{thm}{Theorem}[section]

\newtheorem{lem}[thm]{Lemma}
\newtheorem{prop}[thm]{Proposition}

\theoremstyle{definition}
\newtheorem{defn}[thm]{Definition}
\newtheorem{ex}[thm]{Example}
\newtheorem{question}{Question}
\theoremstyle{remark}
\newtheorem{rem}{Remark}
\newtheorem{nota}{Notation}

\newcommand{\rar}{\rightarrow}

\newcommand{\surjects}{\twoheadrightarrow}

\def\C{{\mathcal C}}
\def\implies{\ifmmode\Rightarrow \else
        \unskip${}\Rightarrow{}$\ignorespaces\fi}
\def\fm{{\mathfrak m}}

\def\VaVa{{\mathcal V}\kern-5pt {\mathcal V}}
\def\lcm{\mathrm{lcm}}
\def\indeg{\mathrm{indeg}}
\begin{document}

\title[The Valabrega-Valla module of monomial ideals]{the Valabrega-Valla module of  monomial ideals}
\author[A. Nasrollah Nejad and  A. A. Yazdan Pour]{abbas Nasrollah nejad and  {ali akbar} {yazdan pour}}
\address{department of mathematics\\ institute for advanced studies in basic sciences (IASBS)\\ p.o.box 45195-1159 \\ zanjan, iran}
\email{abbasnn@iasbs.ac.ir, yazdan@iasbs.ac.ir}

%\thanks{* Corresponding author}

\begin{abstract}
 In this paper, we focus on the initial degree and the vanishing of the Valabrega-Valla module of a pair of monomials ideals $J\subseteq I$ in a polynomials ring over a field $\mathbb{K}$. We prove that the initial degree of this module is bounded above by the maximum degree of a minimal  generators of $J$. For edge ideals of graphs, a complete characterization of the vanishing of the Valabrega-Valla module is given. For higher degree ideals, we find classes which the Valabrega-Valla module vanishes. For the case that $J$ is the facet ideal of a clutter $\mathcal{C}$ and  $I$ is the defining ideal of singular subscheme of $J$, the non-vanishing of this module is investigated in terms of the combinatorics of $\mathcal{C}$.   Finally, we describe the defining ideal of the Rees algebra of $I/J$      
provided that the Valabrega-Valla module is zero.
\end{abstract}
\maketitle

\section*{introduction}
Let $R$ be a commutative Noetherian ring and $J\subseteq I$ ideals in $R$. The \textit{Valabrega-Valla module} of $I$ with respect to $J$ is the graded module 
\begin{equation*}\label{vava}
\VaVa_{J\subseteq I}:=\bigoplus_{t\geq 1} \frac{J\cap
	I^t}{JI^{t-1}}.
\end{equation*}
The Valabrega-Valla module  appeared in~\cite{VaVa}, where its vanishing gives a criterion for Cohen-Macaulayness of the associated graded ring of  the $\fm$-primary ideal $I$  provided that $J$ is a minimal reduction of  $I$  in the Cohen-Macaulay local ring $(R,\fm)$. Later the first author and A. Simis proved that, if $I$ has a regular element modulo $J$, the Valabrega-Valla module is the torsion of the Aluffi algebra~\cite[Proposition  2.5]{AA}. The latter algebra is the algebraic version of  characteristic cycles in intersection theory in the hypersurface case, hence it is interesting for geometric purposes. Dealing directly with the Valabrega-Valla module makes the structure of the Aluffi algebra itself sort of invisible. By the structure of the Aluffi algebra, $\VaVa_{J\subseteq I}=\{0\}$  if and only if  the Aluffi algebra is isomorphic with the Rees algebra of $I/J$. 

In geometric setting, the vanishing of the Valabrega-Valla module is crucial in the intersection theory of regular and linear embedding. More precisely, let $ X\stackrel{i}\hookrightarrow Y \stackrel{j}\hookrightarrow Z$
be closed embeddings of schemes with $\mathcal{J}\subseteq \mathcal{I}\subseteq \mathcal{O}_Z$, the ideal sheaves of $Y$ and $X$
in $Z$, respectively. The embedding $i$ is said to be $\mathrm{linear}$  if every (not necessarily closed)
point $x\in X$ admits  an affine neighborhood $U$ such that the ideal $\mathcal{I}_U/\mathcal{J}_U$ is of linear type in $\mathcal{O}_{Z,U}/\mathcal{J}_U$. A $\mathrm{ regular}$ embedding is defined similarly in terms of regular sequences.
The result of \cite[Theorem 9.2]{Fulton} can be translated to the fact that if $i$ and $j$ are both regular embeddings, then
for all sufficiently large $t$, every point $x\in X$ admits  an affine neighborhood $U$ such that
$\VaVa_{\mathcal{J}_U\subseteq \mathcal{I}_U }=\{0\}$.
More generally, it is shown in \cite[Theorem 1]{Keel} that the same result holds as long as $i$ is a linear
 embedding and $j$ is a regular embedding. 
Thus, under such strong hypothesis, if $Z$ is a regular scheme, the Valabrega-Valla module of $X$ in $Y$, $\VaVa_{X\hookrightarrow Y}=\bigoplus_{t\geq 1} \mathcal{J}\cap \mathcal{I}^t/\mathcal{J}\mathcal{I}^{t-1}$  is zero locally on affine pieces. However, even when $Y$ is a hypersurface embedded in projective space $Z=\mathbb{P}^n$, and $X$ is its singular subscheme, the embedding $X\hookrightarrow Y$ may fail to be linear if $Y$ is non-smooth. This was the main motivation in \cite{AA}, where a detailed analysis was carried assuming an affine situation,
in terms of the relation type number of $I/J$ over $R/J$ and the Artin-Rees number of $J$ relative to $I$.

The vanishing of $\VaVa_{J\subseteq I}$ has close relation with the theory of $I$-standard base (in the sense of Hironaka) which is an essential problem in the resolution of singularities. Indeed, 
 $\VaVa_{J\subseteq I}=\{0\}$ if and only if the tangent  cone $\mathrm{Spec}(\mathrm{gr}_{I/J}(R/J))$ of $Y$ in $X$ is isomorphic with $\mathrm{Spec}(\mathrm{gr}_{I}(R)/J^*)$, where $J^*$ is the form ideal generated by an $I$-standard base of order one in $\mathrm{gr}_{I}(R)$~\cite[Theorem 1.1]{VaVa}.  

The necessary and sufficient conditions for the vanishing of $\VaVa_{J\subseteq I}$ is given  in terms of the first syzygy module of the form ideal $J^*$ in the associated graded ring of $I$~\cite[Theorem 1.2]{AZ1}.  For the case that $J$ is   linear determinantal ideals (rational normal scrolls and alike) or ideal of projective points and  $I$ stands for the Jacobian ideal of $J$ the vanishing problem of $\VaVa_{J\subseteq I}$ is studied in~\cite{AR,AZ}.   

In this paper, we focus on the problem of  the vanishing of the Valabrega-Valla module for a pair of monomial ideals. In the case that $J$ is a edge ideal of a simple graph and $I$ is the Jacobian ideal of $J$, which is also a monomial ideal, the vanishing of $\VaVa_{J \subseteq I}$ characterized combinatorially~\cite[Theorem 3.3]{AR}. The outline of the paper is as follows.   

In section~\ref{c1}, we describe the basic definitions and preliminaries which are used in the sequel, including the Aluffi and Rees algebras, the Valabrega-Valla module,  the Artin-Rees number and the relation type number. As a basic fact, we realize that the vanishing of $\VaVa_{J\subseteq I}=\oplus_{t\geq 1}J\cap I^t/JI^{t-1}$ reduces to the equality 
$J\cap I^t=JI^{t-1}$ for finitely many $t$. 

Sections~\ref{c2} and~\ref{c3} are devoted to be the combinatorial core of this work. One of the main theorem in section~\ref{c2} states that  if $J\subseteq I$ are monomial ideals and $\VaVa_{J\subseteq I}\neq \{0\}$, then the initial degree of $\VaVa_{J\subseteq I}$ is bounded above by the maximum degree of a minimal generators of $J$  (Theorem~\ref{vvmon}). Then for the case that $J\subseteq I$ are edge ideals of graphs, in Proposition~\ref{graph vv} we compute the $\indeg(\VaVa_{J\subseteq I})$  precisely in terms of the combinatorics of the associated graphs. In the last part of this section, we prove that if $\mathcal{C}$ is a complete  $d$-partite $d$-uniform clutter and $\mathcal{C}'$ is a subclutter of $\mathcal{C}$, then the corresponding Valabrega-Valla module of $I(\mathcal{C}')\subseteq I(\mathcal{C}
)$ is zero (Theorem~\ref{d-uni-d-part}).  

In section~\ref{c3}, we introduce  the Valabrega-Valla module of a single ideal $J$, which is by definition the Valabrega-Valla module of the pair $J\subseteq I$, where $I$ is the Jacobian ideal of $J$. For the facet ideal $J$ of a clutter $\mathcal{C}$, we find a non-zero component of the Valabrega-Valla module of $J$ in terms of the combinatorics of $\mathcal{C}$ (Theorem~\ref{ttorsion}). This result is a generalization of one direction of~\cite[Theorem 3.3]{AR}. 

In the last section, we give a presentation for  the Rees algebra of $I/J$ when $I$ is the edge ideal of a graph or the facet ideal of a complete $d$-partite $d$-uniform clutter and $J$ is an appropriate ideal in $I$ and $\VaVa_{J\subseteq I}=\{0\}$. We close the paper by posing some research problems related to this subject.

\section{The vanishing of the Valabrega-Valla module}\label{c1}
Let $R$ be a Noetherian ring and $J\subseteq I$  ideals of $R$. There is a natural surjective $R/J$-algebra homomorphism from the Aluffi algebra $\mathcal{A}_{R/J}(I/J)$ to the Rees algebra $\mathcal{R}_{R/J}(I/J)$
\begin{equation}\label{Aluf-Rees}
\mathcal{A}_{R/J}(I/J)\simeq\bigoplus_{t\geq 0} I^t/JI^{t-1}\surjects \mathcal{R}_{R/J}(I/J)\simeq\bigoplus_{t\geq 0} I^t/J\cap I^t.
\end{equation}

The kernel of this homomorphism is so called the \textit{Valabrega-Valla module} of $I$ with respect to $J$ and is denoted by $\VaVa_{J\subseteq I}$. Indeed,
\begin{equation*}\label{vava2}
\VaVa_{J\subseteq I}:=\bigoplus_{t\geq 1} \frac{J\cap
	I^t}{JI^{t-1}}.
\end{equation*} 

If  $I$ has a regular element modulo $J$, then $\VaVa_{J\subseteq I}$ is $R/J$-torsion of the Aluffi algebra. We will see later that in order to check $\VaVa_{J\subseteq I}=\{0\}$ (i.e., $J\cap I^t=JI^{t-1}$ for all $t\geq 1$), it is enough to show that $J\cap I^t=JI^{t-1}$ for finitely many $t$. 

Given ideals $J,I\subset R$ the \textit{Artin--Rees number} $\mathrm{AR}(J,I)$
of $J$ relative to $I$ is the integer
$$\min\{k\geq 0\colon \quad J\cap I^t=(J\cap I^k)I^{t-k},\;\forall\; t\geq k\}.
$$
Let $\mathfrak{a}$ be an ideal of a ring $A$. The Rees algebra of $\mathfrak{a}$ is defined by $\mathcal{R}_A(\mathfrak{a})=A[\mathfrak{a}t]\subset A[t]$. Let $a_1,\ldots,a_m$ be a minimal generating set for $\mathfrak{a}$. Consider the polynomial ring $A[\mathbf{T}]=R[T_1,\ldots,T_m]$. There is a natural surjective $R$-algebra homomorphism $\psi\colon R[\mathbf{T}]\rar \mathcal{R}_A(I)$ which sends $T_i$ to $a_it$. The kernel of $\psi$ is called the \textit{defining ideal} of the Rees algebra of $\mathfrak{a}$. The \textit{relation type number} $\mathrm{rt}(\mathfrak{a})$ of $\mathfrak{a}$ is the largest degree of any minimal system of homogeneous generators of the kernel $\psi$. Note that this notion is independent of the set of generators  of $\mathfrak{a}$.  

Assume that $I/J$ has regular elements over $R/J$. By  \cite[Corollary 2.6]{AA}, the module of Valabrega-Valla is the zeroth local cohomology of ${\mathcal A}_{R/J}(I/I)$ with respect to $I/J$. In particular, there exists an integer $d\geq 0$ such that $I^d(J\cap I^t)\subseteq JI^{t-1}$ for all $t\geq 1$.  One of the possible such $d$ is $\mathrm{AR}(J,I)-1$ and  $\VaVa_{J\subseteq I}=\{0\}$ if and only if $\mathrm{AR}(J,I)=1$~\cite[Proposition 2.15]{AA}. 
\begin{ex}
Let $J\subset R=\mathbb{K}[x_1,\ldots,x_n],\  (n\geq 3)$ denote the ideal generated by all squarefree monomials in degree $2$, i.e., $J=(x_ix_j \colon\  1\leq i<j\leq n)$. The ideal $J$ is the defining ideal of $n$ coordinate points in $\mathbb{P}^{n-1}$. The Jacobian matrix of $J$ is of the form 
\[
\Theta=\left[\begin{array}{c|ccccc}
x_2&x_1&0&0&\ldots&0\\
x_3&0&x_1&0&\ldots&0\\
\vdots&\vdots&\vdots&\vdots&\ldots&\vdots\\
x_n&0&0&0&\ldots&x_1\\
\hline
0&&&&&\\
\vdots&&&\Theta'&&\\
0&&&&&\\
\end{array}
\right]
,\]
where $\Theta'$ is the Jacobian matrix of the ideal $J'=(x_ix_j \colon \ 2\leq i<j\leq n)$.  
By induction on $n$, we show that the ideal of $(n-1)$-minors of $\Theta$ is the $(n-1)^{th}$-power of the irrelevant maximal ideal $\fm=(x_1,\ldots,x_n)$. For $n=3$, clearly $I_2(\Theta)=(x_1,x_2,x_3)^2$. By induction hypothesis, $I_{n-2}(\Theta')=\fm_1^{n-2}$, where $\fm_1=(x_2,\ldots,x_n)$. Thus $(x_2,\ldots,x_n)\fm_1^{n-2}\subset I_{n-1}(\Theta)$. Therefore by changing the role of $x_1$ by $x_i$ and using the argument as in Example~\ref{Jacobmax}(b) we may conclude that $I_{n-1}(\Theta)=\fm^{n-1}$. Then Example 2.19 in 
~\cite{AA} implies that $\mathrm{AR}(J,I)=1$ while the following discussion shows that the  relation type number of $I/J$ is $2$. 

Note that the Jacobian ideal $I=(J,\fm^{n-1})$ is generated minimally by $J$ and monomials $x_1^{n-1},\ldots,x_n^{n-1}$. Set $\mathfrak{a}=(x_1^{n-1},\ldots,x_n^{n-1})$.  
Let $G$ be a simple graph which consists of  a complete graph on vertex set $\{T_1,\ldots,T_n\}$ and each vertex $T_i$ has $(n-1)$ whiskers $x_1,\ldots, \hat{x}_{i},\ldots, x_n$. 
We claim that 
\[\mathcal{R}_{R/J}(I/J)\simeq \mathcal{R}_{\bar{R}}(\bar{\mathfrak{a}})\simeq \frac{R[T_1,\ldots, T_n]}{(J,I(G))},   \] 
where $I(G)$ is the edge ideal of the graph $G$. Clearly, $(J,I(G))$ is included in the defining ideal $\mathcal{J}$ of the Rees algebra of $\bar{\mathfrak{a}}$. Conversely,  let $F(T_1,\ldots,T_n)\in \mathcal{J}$  be a homogeneous polynomial of degree $r\geq 1$ and $u=x_1^{\alpha_1}\cdots x_n^{\alpha_n}T_1^{\beta_1}\cdots T_n^{\beta_n}$ be a monomial in the support of $F$. Since $F(x_1^{n-1},\ldots, x_n^{n-1})\in J$, it follows that $(x_1^{n})^{\alpha_1+\beta_1}\cdots (x_n^{n})^{\alpha_n+\beta_n}\in J$ and $\beta_k>0$ for some $k$. Hence there exist $i<j$ such that $x_ix_j\mid u(x_1^{n-1},\ldots,x_n^{n-1})$ and $\alpha_i+\beta_i>0,\, \alpha_j+\beta_j>0$. If $\beta_i,\beta_j>0$ or $\alpha_i,\beta_j>0$ or $\alpha_j,\beta_i>0$,  then clearly $u\in I(G)$. Otherwise, either $i\neq k$ or $j\neq k$ and $x_iT_k\mid u$ or $x_jT_k\mid u$. In both cases, we conclude that $u\in I(G)$. Thus $F\in I(G)$ and in particular $\mathrm{rt}(I/J)=2$. 
\end{ex}
The vanishing of the Valabrega-Valla module has close relation with the Artin-Rees number of $J$ relative to $I$ and the relation type number of $I/J\subseteq R/J$. We will show that these numbers have relation with the initial degree of $\VaVa_{J\subseteq I}$. 
%\begin{defn}\label{NVaVa}
%Let $J\subseteq I$ be ideals in the ring $R$ and $t\geq 1$ a positive integer. We say that $J\subseteq I$ is \textit{$\VaVa_{ \leq r}$-torsion-free}, if $J\cap I^n=JI^{n-1}$ for all $1\leq n\leq r$. 
%\end{defn}

For a given graded ring $A$ and a graded $A$-module $0\neq M=\oplus_{i\in \mathbb{N}} M_i$, the \textit{initial degree} of $M$ is defined by 
\[\mathrm{indeg}(M)=\min \{ i\colon \ M_i\neq 0\}.\]

If $J\cap I^n=JI^{n-1}$ for all $1\leq n\leq \ell$ where $\ell$ is the Artin-Rees number of $J$ relative to $I$, then $\VaVa_{J\subseteq I}=\{0\}$~ \cite[Lemma 2.16]{AA}. Thus if $\VaVa\neq \{0\}$, then 
\begin{equation}\label{indef-art-rtype}
\mathrm{indeg}(\VaVa_{J\subseteq I})\leq \mathrm{AR}(J,I)\leq \mathrm{rt}(I/J),
\end{equation}
where the last inequality comes from~\cite[Theorem 2]{FPV}.

Let $I$ be an ideal in the ring $R$. Recall that an ideal $J \subseteq I$ is called a \textit{reduction} of $I$, if $JI^n =I^{n+1}$, for sufficiently large $n$. For a reduction $J$ of $I$, let
\[{r}_J(I)= \min \{t \colon \quad JI^n=I^{n+1}, \text{ for all } n \geq t \} \]
be the \textit{reduction number} of $I$ relative to $J$. It is obvious from definition that $\indeg(\VaVa_{J\subseteq I})\leq r_J(I)$ provided that $\VaVa_{J\subseteq I}\neq \{0\}$. Therefore, in the case that $J\subseteq I$ is a reduction of $I$ and $\VaVa_{J\subseteq I}\neq \{0\}$, one has 
\[\indeg(\VaVa_{J\subseteq I})\leq \min\{\mathrm{AR}(J,I),\, {r}_J(I)  \} . \]

\begin{prop}\label{dpower} 
Let $J\subset R=\mathbb{K}[\mathbf{x}]$ be a homogeneous ideal and $r\geq 1$ be an integer such that $J\subseteq \fm^r$, where $\fm=(\mathbf{x})$. 
\begin{itemize}
\item[{\rm (a)}] If $\VaVa_{J\subseteq \fm^r}=\{0\}$, then $\mathrm{indeg}(J)=r$. 
\item[{\rm (b)}] If $J$ is generated by some forms of degree $r$, then $\VaVa_{J\subseteq \fm^r}=\{0\}$.
\end{itemize}
\end{prop}

\begin{proof}

Since $J\subseteq \fm^r$, it follows that $r\leq \mathrm{indeg}(J)$.  Let $f\in J$ be a homogeneous polynomial and  $t\geq 1$ be an integer such that $rt\geq \deg(f)$. Then $fx_1^{rt-\deg(f)}$ belongs to $J\cap \fm^{rt}=J\fm^{rt-r}$. Thus $rt=\deg(fx_1^{rt-\deg(f)})\geq rt-r+\mathrm{indeg}(J)$. This completes the proof of (a). 
The statement (b) follows from \cite[Example 2.19]{AA}. 
\end{proof}

\begin{ex}
Let $R$ be a commutative Noetherian  ring and $J \subseteq I$ be ideals in $R$ such that $\VaVa_{J\subseteq I}\neq \{0\}$.
\begin{itemize}
\item[(a)] If $\mathfrak{q}$ is an ideal in $R$ such that $\mathfrak{q}^r \subseteq J \subseteq I \subseteq \mathfrak{q}^s$, for some $r \geq s \geq 1$, then $\indeg(\VaVa_{J\subseteq I})\leq \lceil r/s \rceil$. Because for all $t \geq \lceil r/s \rceil$, we have $J\cap I^t=I^t$, hence 
\[J\cap I^t=I^t=I^{t-\lceil r/s \rceil}I^{\lceil r/s \rceil}=I^{t-\lceil r/s \rceil}(J\cap I^{\lceil r/s \rceil}),\]
for all $t \geq \lceil r/s \rceil$. Therefore, the Artin--Rees number of $J$ realative to $I$ is bounded above by $\lceil r/s \rceil$. The result follows from the fact that $\indeg(\VaVa_{J\subseteq I})$ is bounded above by the Artin-Rees number.

\item[(b)] Let $(R,\fm)$ be a local ring and $J\subseteq I$ be ideals such that $\dim(R/J)=0$. Then there exists an integer $r\geq 1$ such that $\fm^{r}\subseteq J \subseteq I$. It follows from (a) that $\indeg(\VaVa_{J\subseteq I})\leq r$.

\item[(c)] There exists $r \geq 1$ such that $(\sqrt{J})^r \subseteq J \subseteq \sqrt{J}$. Hence $\indeg(\VaVa_{J\subseteq \sqrt{J}})\leq r$, by (a).
\item[(d)]  Let $J\subseteq I$ be ideals in a local ring $(R,\fm)$ such that $I$ is $\fm$-primary and $R/J$ is Cohen-Macaulay of dimension one. Then $\indeg(\VaVa_{J\subseteq I})\leq e(R/J)$ by~\cite[Lemma 6.3]{FPV}. Here $e(R/J)$ denotes the multiplicity of $R/J$.

%\item[(d)] Recall that an ideal $J \subseteq I$ is called a reduction of $I$, if $JI^n =I^{n+1}$, for sufficiently large $n$. For a reduction $J$ of $I$, let
%\[r_J(I)= \min \{t \colon \quad JI^n=I^{n+1}, \text{ for all } n \geq t \} \]
%be the reduction number of $I$ relative to $J$. It is obvious from definition that $N\VaVa_{J\subseteq I}\leq r_J(I)$, whenever $J \subseteq I$ is a reduction of $I$. 
\end{itemize}
\end{ex} 

\section{The initial degree of the Valabrega-Valla module  of monomial ideals}\label{c2}
Let $J$ be a monomial ideal in the polynomial ring $R=\mathbb{K}[{\mathbf x}]$.  Let $\mathcal{G} (J)$ denotes its unique minimal set of generators, we define 
$$t_0(J) := \max \left\{ \deg(F) \colon \quad F \in \mathcal{G}(J) \right\}.$$

Let $J\subseteq I$  be monomial ideals in $R$. 
In view of~\eqref{indef-art-rtype}, we know that if $\VaVa_{J\subseteq I}\neq \{0\}$, then there exists $\ell \leq \mathrm{AR}(J,I)\leq \mathrm{rt}(I/J)$ such that $(\VaVa_{J\subseteq I})_{\ell}\neq \{0\}$. In general finding the Artin-Rees number or even an upper bound for this number  is not easy even for monomial ideals. In the following, we show that if $\VaVa_{J\subseteq I}\neq \{0\}$, then there exists $\ell\leq \min\{t_0(J),\, \mathrm{AR}(J,I)\}$ such that $(\VaVa_{J\subseteq I})_{\ell}\neq \{0\}$ (Theorem~\ref{vvmon}). To prove this, we need the following easy lemma.  
\begin{lem}\label{easy lemma}
	Let $m,g_1, \ldots, g_s$ be polynomials in $R$ such that $m$ is a non-constant polynomial dividing $g_1 \cdots g_s$. Then there exist $r \leq s$, a permutation $\sigma$ of $\{1, \ldots, r\}$, and non-constant polynomials $u_1, \ldots, u_r$ such that $u_i \mid g_{\sigma(i)}$ and $m=u_1 \cdots u_r$.
\end{lem}

\begin{proof}
	We use induction on $s$ to obtain the assertion. The statement is clear for $s=1$. Let $s>1$ and the assertion holds for all polynomials $m, g_1, \ldots, g_{s-1}$ with $m$ a nono-constant polynomial dividing $g_1 \cdots g_{s-1}$. Let $m,g_1, \ldots, g_s$ be polynomials in $R$ such that $m$ is a non-constant polynomial dividing $g_1 \cdots g_s$ and $m=f_1 \cdots f_k$ be the decomposition of $m$ into prime components. Without loss of generality, assume that $f_1 \mid g_1$. Let $u_1=\gcd(m,g_1)$ and $M={m}/{\gcd(m,g_1)}$. If $M$ is a constant polynomial, then we are done. Otherwise, $M$ is a non-constant polynomial dividing $g_2 \cdots g_s$. By induction hypothesis, there exist $r \leq s$, a permutation $\sigma$ of $\{2, \ldots, r\}$, and non-constant polynomials $u_2, \ldots u_r$ such that $u_i \mid g_ {\sigma(i)}$ ($i=2, \ldots, r$) and $M=u_2 \cdots u_s$. Then the polynomials $u_1, \ldots, u_r$ satisfy the required properties.
\end{proof}

\begin{thm}\label{vvmon}
Let $J \subseteq I$ be monomial ideals in $R=\mathbb{K}[\mathbf{x}]$.  If $\VaVa_{J\subseteq I}\neq \{0\}$, then $$\indeg(\VaVa_{J\subseteq I})\leq \min\{t_0(J),\ \mathrm{AR}(J,I)\}.$$
\end{thm}

\begin{proof}
Let $s_0=\indeg(\VaVa_{J\subseteq I})$. By virtue of~\eqref{indef-art-rtype}, it is enough to show that $s_0\leq t_0(J)$.
Clearly, $J I^{s_0-1}\subseteq  J \cap I^{s_0}$. Hence there exists a generator $g \in J \cap I^{s_0}$, such that $g \notin J I^{s_0-1}$. Let write $g=g_1 \cdots g_{s_0}$ with  $g_i \in I$. 
	
	%If $g_i \in J$, for some $1 \leq i \leq s_0$,  then $g= g_i \left( g_1 \cdots g_{i-1} g_{i+1} \cdots g_{s_0} \right) \in J I^{s_0-1}$, which is a contradiction. However, the monomial $g$ is in $J$, hence there exists a monomial $m \in \mathcal{G}(J)$ such that, $m | g$ and $m \nmid g_i$, for $i=1, \ldots s_0$. 
	
Since $g \in J$, there exists a monomial $m \in \mathcal{G}(J)$ such that, $m | g$. We use the Lemma~\ref{easy lemma} to obtain $s \leq s_0$, a permutation $\sigma$ of $\{1, \ldots, s \}$, and the monomials $u_1, \ldots, u_s$ such that $u_i \mid g_{\sigma(i)}$ and $m=u_1 \cdots u_s$. Without loss of generality, assume that $\sigma(i) =i$ for all $i=1, \ldots, s$. Then, $s \leq \deg (m) \leq t_0(J)$, and $g_1 \cdots g_s \in J \cap I^{s}$. 
If $g_1 \cdots g_s \in J I^{s-1}$, then $g = \left( g_1 \cdots g_s \right) \, \left( g_{s+1} \cdots g_{s_0} \right) \in  J I^{s_0-1}$
which contradicts with our choice of $g$. Thus we have,	$J \cap I^{s} \neq J I^{s-1}.$
The minimality of $s_0$ implies that, $s_0 \leq s \leq t_0(J)$. This completes the proof.
\end{proof}
Let $J\subseteq I$ be monomial ideals generated in degree $2$. Then by virtue of Theorem~\ref{vvmon}, we know that either $\VaVa_{J\subseteq I}=\{0\}$ or $\indeg(\VaVa_{J\subseteq I})=2$. In the following we characterize those pair of ideals $J\subseteq I$ such that $\VaVa_{J\subseteq I}=\{0\}$ in the case that $J$ and $I$ are the edge ideals of some graphs. For this characterization, we need the following definition.  
\begin{defn}\label{embedded subgraphs}
Let $G'$ be a subgraph of $G$.
\begin{itemize}
\item[(i)] the graph $G'$ is called \textit{almost $C_3$-embedded} subgraph of $G$ if for all $3$-cycle $i-j-k-i$ in $G$ with $\{i,j\} \in E(G')$, either $\{i,k\} \in E(G')$ or $\{j,k\} \in E(G')$.
\item[(ii)] the graph $G'$ is called \textit{almost $P_3$-embedded} subgraph of $G$ if for all $3$-path $i'-i-j-j'$ in $G$ with $\{i,j\} \in E(G')$ and $\{i',j'\} \notin E(G)$, either 
\begin{itemize}
\item[(a)] $\{i',i\} \in E(G')$, or
\item[(b)] $\{j',j\} \in E(G')$, or
\item[(c)] $\{i',j\} \in E(G')$ and $\{i,j'\} \in E(G)$, or
\item[(d)] $\{i,j'\} \in E(G')$ and $\{i',j\} \in E(G)$.
\end{itemize}
\end{itemize}
\end{defn}

\begin{ex} \label{examples of embedded graphs} \mbox{}
\begin{itemize}
\item[(i)] Since the bipartite graph does not have any cycle of odd length, it follows that every subgraph of a bipartite graph is almost $C_3$-embedded subgraph.
\item[(ii)] If $G$ is a complete graph or complete bipartite graph, then every subgraph of $G$ is almost $P_3$-embedded subgraph.
\item[(iii)] Every subgraph of a complete bipartite graph is both almost $C_3$-embedded and almost $P_3$-embedded subgraph.
\end{itemize}
\end{ex}

\begin{prop} \label{graph vv}
Let $G'$ be a subgraph of $G$, $J=I(G')$ and $I=I(G)$.
\begin{itemize}
\item[\rm (i)] $\VaVa_{J \subseteq I} =\{0\}$ if and only if $G'$ is both almost $C_3$-embedded and almost  $P_3$-embedded subgraph of $G$.
\item[\rm (ii)] If $G=K_n$ is the complete graph on the vertex set $[n]=\{1, \ldots, n\}$, then $\VaVa_{J \subseteq I} =\{0\}$ if and only if for all $\{i,j\} \in E(G')$, $N_{G'}(i) \cup N_{G'}(j) =[n]$.
\item[\rm (iii)] If $G$ is the complete bipartite graph, then $\VaVa_{J \subseteq I} =\{0\}$.
\end{itemize}
\end{prop}

\begin{proof}
By Theorem~\ref{vvmon}, $\VaVa_{J \subseteq I} =\{0\}$ if and only if $J \cap I^2 \subseteq JI$. We know that $J \cap I^2$ is generated by monomials $\lcm (\textbf{x}_e, \textbf{x}_{e'}\textbf{x}_{e''})$ where $e \in E(G')$ and $e', e'' \in E(G)$. If $e \cap (e' \cup e'')\subseteq e'$ or $e \cap (e' \cup e'')\subseteq e''$, then clearly $\lcm (\textbf{x}_e, \textbf{x}_{e'}\textbf{x}_{e''}) \in  JI$.

(i) Assume that $\VaVa_{J \subseteq I} =\{0\}$. We show that $G'$ is both almost $C_3$-embedded and almost $P_3$-embedded subgraph of $G$.
Let $i-j-k-i$ be a $3$-cycle in $G$ with $\{i,j\} \in E(G')$ and $\{i,k\}, \{j,k\} \notin E(G')$. Then $x_ix_jx_k^2 =\lcm (x_ix_j, x_ix_k\ x_jx_k) \in J \cap I^2 \setminus JI$. This shows that $G'$ is a almost $C_3$-embedded subgraph of $G$. Similarly, let $i'-i-j-j'$ be a $3$-path  in $G$ with $\{i,j\} \in E(G')$ and $\{i',j'\} \notin E(G)$. Then 
\[ x_ix_jx_{i'}x_{j'} = \lcm (x_ix_j, x_ix_{i'}\ x_jx_{j'}) \in J \cap I^2 \subseteq JI.\]
It follows that one of the conditions (a)-(d) of Definition~\ref{embedded subgraphs}(ii) satisfies.

Conversely, assume that $G'$ is both almost $C_3$-embedded and almost $P_3$-embedded subgraph of $G$. By the above discussion, it is enough to show that $\lcm (\textbf{x}_e, \textbf{x}_{e'}\textbf{x}_{e''}) \in  JI$, for all $e \in E(G')$, $e', e'' \in E(G)$ with $|e \cap e'| = |e \cap e''| =1$ and $e \cap e' \neq e \cap e''$. Without loss of generality, assume that $e=\{i,j\}$, $e'=\{i,i'\}$ and $e''=\{j,j'\}$. Then $\lcm (\textbf{x}_e, \textbf{x}_{e'}\textbf{x}_{e''})=x_i x_j x_{i'}x_{j'}$. If $i'=j'$, then $i-j-i'-i$ is a $3$-cycle in $G$ with $\{i,j\} \in E(G')$, so our assumption implies that $\lcm (\textbf{x}_e, \textbf{x}_{e'}\textbf{x}_{e''})=x_i x_j x_{i'}^2 \in JI$. Otherwise, $i'-i-j-j'$ is a $3$-path in $G$ with $\{i,j\} \in E(G')$. If $\{i',j'\} \in E(G)$ then clearly  $\lcm (\textbf{x}_e, \textbf{x}_{e'}\textbf{x}_{e''})=x_i x_j x_{i'}x_{j'} \in JI$. Otherwise, by our assumption, one of of the conditions (a)-(d) of Definition~\ref{embedded subgraphs}(ii) satisfies. This is equivalent to  say that $\lcm (\textbf{x}_e, \textbf{x}_{e'}\textbf{x}_{e''})=x_i x_j x_{i'}x_{j'} \in JI$. This completes the proof of (i).

(ii) In view of part (i), it is enough to show that the following statements are equivalent:
\begin{enumerate}
\item $G'$ is both almost $C_3$-embedded and almost $P_3$-embedded subgraph of $K_n$.
\item for all $\{i,j\} \in E(G')$, $N_{G'}(i) \cup N_{G'}(j) =[n]$.
\end{enumerate}
If $G'$ is both almost $C_3$-embedded and almost $P_3$-embedded subgraph of $K_n$, $\{i,j\} \in E(G')$ and $k \in [n]$, then $i-j-k-i$ is a $3$-cycle in $K_n$. So by our assumption, either $\{i,k\} \in E(G')$ or $\{j,k\} \in E(G')$, i.e. $k \in N_{G'}(i) \cup N_{G'}(j)$. Conversely, assume that $N_{G'}(i) \cup N_{G'}(j) =[n]$, for all $\{i,j\} \in E(G')$. Let $i-j-k-i$ be a $3$-cycle in $G$ with $\{i,j\} \in E(G')$. Then by our assumption, $k \in N_{G'}(i) \cup N_{G'}(j)$. It follows that either $\{i,k\} \in E(G')$ or $\{j,k\} \in E(G')$. Hence $G'$ is almost $C_3$-embedded subgraph of $G=K_n$. Since $G=K_n$ is a complete graph, it is obvious that $G'$ is almost $P_3$-embedded subgraph of $G$. Thus (1) and (2) are equivalent.

(iii) By Example~\ref{examples of embedded graphs}(iii), every subgraph of a complete bipartite graph is both almost $C_3$-embedded and almost $P_3$-embedded subgraph. The result follows from (i).
\end{proof}

\begin{ex}\label{biparti-subgrpah}
Let $G=K_n$ be complete graph on the vertex set $[n]$ and $G'$ be a  subgraph of $G$ that $\VaVa_{I(G') \subseteq I(G)} =\{0\}$. Let $\chi$ be a minimal vertex coloring of $G'$ and $A_1, \ldots, A_{\chi(G')}$ be the class coloring of the vertex set of $G'$, by which we mean $A_i  = \{ u \in [n] \colon \; \chi(u)=i\}$. Then by minimality of coloring, for all $1 \leq k \neq k' \leq \chi(G')$, there exist $u \in A_k$ and $v \in A_{k'}$ such that $\{u,v\} \in E(G')$. Then Theorem~\ref{graph vv}(ii) implies that $u$ (respectively $v$) is adjacent to all vertices in $A_{k'}$ (respectively $A_k$). This shows that $G'$ is a complete multipartite graph with $A_1 \cup \cdots \cup A_{\chi(G')}=[n]$. Conversely, if $G'$ is a complete multipartite graph whose the vertex set is $[n]$, then clearly $N(u) \cup N(v)=[n]$ for all $\{u,v\} \in E(G')$. Consequently, $\VaVa_{I(G') \subseteq I(G)} =\{0\}$ if and only if $G'$ is a complete multipartite graph on the vertex set $[n]$. Otherwise, $\indeg(\VaVa_{I(G') \subseteq I(G)}) =2$ by Theorem~\ref{vvmon}.
\end{ex}

In the following we find  generalizations of Proposition~\ref{graph vv}(ii) and (iii) for a pair of ideals generated in degree $d>2$. To achieve  this, we replace the graphs by more general structure, called  clutters. 

\begin{defn}
Let $[n]=\{ 1, \ldots, n\}$. A \textit{clutter} $\mathcal{C}$ on vertex set $[n]$ is a collection of subsets of $[n]$, called \textit{circuits} of $\mathcal{C}$, such that $e_1 \nsubseteq e_2$, for all $e_1$ and $e_2$ in $\mathcal{C}$. We call
$V(\C)=\cup_{F\in \C} F$ the \emph{set of vertices of $\C$}. A \textit{$d$-circuit} is a circuit consisting of exactly $d$ vertices and a clutter is \textit{$d$-uniform}, if every circuit has exactly $d$ vertices.

For a non-empty clutter $\C$ on vertex set $[n]$, we define the ideal $I(\C)$, as follows:
$$I(\C) = \left( \textbf{x}_F \colon \quad F \in \C \right)$$
and we define $I(\varnothing) = 0$. The ideal $I\left( \C \right)$ is called \textit{facet ideal} of $\C$. Here $\textbf{x}_F$ is $\prod_{i\in F}x_i$.

Following \cite{d-partite}, we say that a $d$-uniform clutter $\C$ is \emph{$d$-partite}, if $V(\C)$ can be written as the union of mutually disjoint subsets $V_1, \ldots, V_{d}$, such that each circuit of $\C$ meets each $V_i$ in exactly  one vertex. If moreover, $\C$ contains all $d$-subsets of $V(\C)$ which intersect each $V_i$ in exactly  one vertex, we say that $\C$ is \emph{complete $d$-partite} clutter. The partition $\{V_i\colon \; i \in [d] \}$ as above is called a \emph{$d$-partition} of $\C$.
\end{defn}
 
\begin{thm}\label{d-uni-d-part}
Let $\mathcal{C}$ be a complete $d$-partite $d$-uniform clutter and $\mathcal{C}' \subseteq \mathcal{C}$. Then $\VaVa_{I(\mathcal{C}') \subseteq I(\mathcal{C})} =\{0\}$.
\end{thm}

\begin{proof}
Let $V_1, \ldots, V_d$ denote the $d$-partition of the vertices of $\C$, $J=I(\mathcal{C}')$ and $I = I(\mathcal{C})$.  We have to show that $J \cap I^t = JI^{t-1}$, for all $t \geq 2$. Let $F \in \mathcal{C'}$ and $F_1, \ldots, F_t \in \mathcal{C}$ be $d$-subsets of $[n]$ and put
\[ 
\begin{array}{lcll}
A_1 & = &F \cap F_1,\\
A_2 & = & \left( F \cap F_2 \right) \setminus A_1,\\
& \vdots & \\
A_t &= & \left( F \cap F_t \right) \setminus \left( A_1 \cup A_2 \cup \cdots \cup A_{t-1} \right).
\end{array}
\]
Then, $\lcm(\mathbf{x}_F, \mathbf{x}_{F_1} \cdots \mathbf{x}_{F_1})= \mathbf{x}_{F} \cdot \mathbf{x}_{F_1 \setminus A_1} \cdot \ldots \cdot \mathbf{x}_{F_{t-1} \setminus A_{t-1}} \cdot \mathbf{x}_{F_t \setminus A_t}$. It is enough to show that $\mathbf{x}_{F_1 \setminus A_1} \cdot \ldots \cdot \mathbf{x}_{F_{t-1} \setminus A_{t-1}} \cdot \mathbf{x}_{F_t \setminus A_t} \in I^{t-1}$. 
If $1 \leq i,j,k \leq t$, $1 \leq s \leq d$ be positive integers such that $F_i \cap F_j \cap V_s \neq \varnothing$ and $F_i \cap F_k \cap V_s \neq \varnothing$, then $F_i \cap F_j \cap V_s= F_i \cap F_k \cap V_s$, because $|G \cap V_k|=1$ for all $G \in \mathcal{C}$. Let
\begin{align*}
& T_i=\{t\in [d] \colon \quad V_t \cap A_i \neq \varnothing \}, &1 \leq i \leq t,\\
& B_1=F_t \cap \left( \mathop{\cup}\limits_{s \in T_1}V_s \right), & \\
& B_i=\left( F_t \setminus \left( \mathop{\cup}\limits_{r=1}^{i-1} B_r \right) \right) \cap \left( \mathop{\cup}\limits_{s \in T_i} V_s \right), & 2 \leq i\leq t.
\end{align*}
We claim that 
\begin{itemize}
\item[(a)] $B_i \cap (F_i \setminus A_i) = \varnothing$,
\item[(b)] $B_i \subseteq \left( F_t \setminus \left( \mathop{\cup}\limits_{r=1}^{i-1} B_r \right) \right) \setminus A_t$,
\item[(c)] $|B_i|=|A_i|$.
\end{itemize}
for all $1 \leq i \leq s$.
\begin{itemize}
\item[] \textbf{Proof of the claim.} (a) If $y \in B_i \cap F_i$, then by definition, there exists $s \in T_i$ such that $y \in F_t \cap V_s$. Hence $F_i \cap F_t \cap V_s \neq \varnothing$ and $A_i \cap V_s \neq \varnothing$. It follows that $ A_i \cap V_s = F_i \cap F \cap V_s$. Then by above discussion, $y \in  F_i \cap F_t \cap V_s =  F_i \cap F \cap V_s =A_i \cap V_s$. Thus $y \in A_i$.\\
(b) Let $y \in B_i$. If $y \notin F$, then clearly $y$ belongs to the right side. Otherwise, choose $s \in T_i$ such that $y \in F_t \cap V_s$. Then $F \cap F_t \cap V_s \neq \varnothing$ and $A_i \cap V_s \neq \varnothing$. It follows that $\varnothing \neq A_i \cap V_s = F \cap F_i \cap V_s$. Then by above discussion, $y \in  F \cap F_t \cap V_s =  F \cap F_i \cap V_s =A_i \cap V_s$. Thus $y \in A_i$. This completes the proof of (b).\\
(c) It is easy to check that $|T_i|= |A_1|+ \cdots + |A_i|$, for all $i=1, \ldots, t$. Next, we note that for all $s \in T_i$ one has $|F_t \cap V_s|=1$ and that $F_t \cap V_s \neq F_t \cap V_{s'}$, if $s \neq s' \in T_i$. This shows that 
\begin{equation} \label{|Ft cap Vs|}
|F_t \cap \left( \mathop{\cup}\limits_{s \in T_i} V_s \right)| = |T_i|= |A_1|+ \cdots + |A_i |.
\end{equation}
In particular, $|B_1|=|A_1|$. Assume by induction that $|B_j|=|A_j|$, for all $j=1, \ldots, i$. Then by \eqref{|Ft cap Vs|} we have
\begin{align*}
|B_i| &= |\left( F_t \setminus \left( \mathop{\cup}\limits_{r=1}^{i-1} B_r \right) \right) \cap \left( \mathop{\cup}\limits_{s \in T_i} V_s \right)| \\
& = |T_i|-|B_1|-\cdots-|B_{i-1}| \\
& = |A_1|+ \cdots + |A_i |-|B_1|-\cdots-|B_{i-1}| =|A_i|.
\end{align*}
\end{itemize}
Now we rewrite the lcm as follows
\begin{align*}
\lcm(\mathbf{x}_F, \mathbf{x}_{F_1} \cdots \mathbf{x}_{F_t})& = \mathbf{x}_{F} \cdot \mathbf{x}_{F_1 \setminus A_1} \cdot \ldots \cdot \mathbf{x}_{F_{t-1} \setminus A_{t-1}} \cdot \mathbf{x}_{F_t \setminus A_t}\\
&= \mathbf{x}_{F} \cdot \mathbf{x}_{(F_1 \setminus A_1) \cup B_1} \cdot  \ldots \cdot \mathbf{x}_{(F_{t-1} \setminus A_{t-1}) \cup B_{t-1}} \cdot \mathbf{x}_{F_t \setminus B_1 \setminus \cdots \setminus B_{t-1} \setminus A_t}
\end{align*}
The monomial in the right side belongs to $JI^{t-1}$, by (a)--(c) above.
\end{proof}
Let $\fm^{[d]}$ denotes the $d$th square-free power of the maximal ideal $\fm =(x_1, \ldots,x_n)$ in $\mathbb{K}[\mathbf{x}]$. That is $\fm^{[d]}= \left( \mathbf{x}_F \colon F \subseteq [n], \; |F|=d \right)$.
In the rest of this section, we consider the pair $J\subseteq \fm^{[d]}$ and we characterize those ideals $J\subseteq \fm^{[3]}$ such that $\VaVa_{J\subseteq \fm^{[3]}}=\{0\}$. 
\begin{lem} \label{square-free power of maximal ideal}
Let $J$ be a square-free monomial ideal generated in degree $d$. If $J \cap (\fm^{[d]})^d = J(\fm^{[d]})^{d-1}$, then for all $x_{i_1} \cdots x_{i_d} \in \mathcal{G}(J)$ and  for all $\textbf{\rm \textbf{x}}_G \in \mathcal{G}(\fm^{[d-1]})$, there exists $1 \leq r \leq d$, such that $\textbf{\rm \textbf{x}}_{G \cup \{i_r\}} \in J$.
\end{lem}

\begin{proof}
Assume that $x_{i_1} \cdots x_{i_d} \in \mathcal{G}(J)$ and $\textbf{\rm \textbf{x}}_G \in \mathcal{G}(\fm^{[d-1]})$ be monomials such that the assertion is not true. Then, by our assumption, we have
\[x_{i_1} \cdots x_{i_d} \cdot (\mathbf{x}_G)^d =\lcm (x_{i_1} \cdots x_{i_d}, x_{i_1}\mathbf{x}_G \cdot x_{i_2}\mathbf{x}_G \cdot \ldots \cdot x_{i_d}\mathbf{x}_G) \in J \cap (\fm^{[d]})^d = J(\fm^{[d]})^{d-1}.\]
Hence there exist $\varnothing \neq H \subseteq \{i_1, \ldots, i_d\}$ and $H' \subseteq G$ such that $|H \cup H'|=d$ and $\mathbf{x}_{H \cup H'} \in J$. Since $\mathbf{x}_{G \cup \{i_j\}} \notin J$, for all $1 \leq j \leq r$, one has $|H|>1$. Hence $|\{i_1, \ldots, i_d\} \setminus H|<d-1$ and 
\[\mathbf{x}_{\{i_1, \ldots, i_d\} \setminus H} \cdot \mathbf{x}_{G \setminus H'} \cdot (\mathbf{x}_G)^{d-1} \in (\fm^{[d]})^{d-1}.\]
This is impossible by the definition of $\fm^{[d]}$.
\end{proof}

\begin{rem}
Let $J \subseteq \mathbb{K}[x_1, \ldots, x_n]$ be a square-free monomial ideal generated in degree $d$ and $\fm=(x_1, \ldots, x_n)$. If $\VaVa_{J \subseteq \fm^{[d]}} =\{0\}$, then by Lemma~\ref{square-free power of maximal ideal}, for all $x_{i_1} \cdots x_{i_d} \in \mathcal{G}(J)$ and  for all $(d-1)$-subset $G \subseteq [n]$, there exists $1 \leq r \leq d$, such that $\textbf{\rm \textbf{x}}_{G \cup \{i_r\}} \in J$. \textit{Is the converse of this statement true?}

In the following, we show that the converse is true in the case $d=3$.
\end{rem}

\begin{prop}\label{square-free pairs}
Let $J \subseteq \mathbb{K}[x_1, \ldots, x_n]$ be a square-free monomial ideal generated in degree $3$ and $\fm=(x_1, \ldots, x_n)$. Then the followings are equivalent:
\begin{itemize}
\item[\rm (a)] $\VaVa_{J \subseteq \fm^{[3]}} =\{0\}$,
\item[\rm (b)] for all $x_{i_1} x_{i_2} x_{i_3} \in \mathcal{G}(J)$ and  for all $2$-subsets $G \subseteq [n]$, either $\textbf{\rm \textbf{x}}_{G \cup \{i_1\}} \in J$ or $\textbf{\rm \textbf{x}}_{G \cup \{i_2\}} \in J$ or $\textbf{\rm \textbf{x}}_{G \cup \{i_3\}} \in J$.
\end{itemize}
\end{prop}

\begin{proof}
(a) $\to$ (b): This implication follows from more general case in Lemma~\ref{square-free power of maximal ideal}.

(b) $\to$ (a): By virtue of Theorem~\ref{vvmon}, it is enough to show that $J \cap (\fm^{[3]})^t =J (\fm^{[3]})^{t-1}$, for $t=2,3$. Since $J \subseteq \fm^{[3]}$, we conclude that $J (\fm^{[3]})^{t-1} \subseteq J \cap (\fm^{[3]})^t$, for all $t$. So it is enough to show the other inclusion. Note that $J \cap (\fm^{[3]})^t$ is generated by $\lcm (\mathbf{x}_F, \mathbf{x}_{F_1} \cdots \mathbf{x}_{F_t})$, where $\mathbf{x}_F \in \mathcal{G}(J)$ and $F_1, \ldots, F_t$ are $3$-subsets of $[n]$. 

Let $F=\{i_1, i_2, i_3\}$, $F_1, \ldots, F_t$ be $3$-subsets of $[n]$ such that $\mathbf{x}_F \in \mathcal{G}(J)$ and $u = \lcm (\mathbf{x}_F, \mathbf{x}_{F_1} \cdots \mathbf{x}_{F_t})$.
If $F \cap \left( F_1 \cup \cdots \cup F_t \right) \subseteq F_i$ for some $1 \leq i \leq t$, then 
\[u= \mathbf{x}_F \cdot \frac{\mathbf{x}_{F_1} \cdot \ldots \cdot \mathbf{x}_{F_t}}{\mathbf{x}_{F_i}} \cdot \mathbf{x}_{F_i \setminus F \cap \left( F_1 \cup \cdots \cup F_t \right)} \in J(\fm^{[3]})^{t-1}.
\]
So assume that $F, F_1, \ldots, F_t$ do not satisfy in this condition. We consider the following cases:

\textbf{Case 1.} $|F \cap (F_1 \cup \cdots \cup F_t)|=2$.
\begin{itemize}
\item[] Since $F, F_1, \ldots, F_t$ do not satisfy in the above condition, without loss of generality, we may assume that $F \cap F_1=\{i_1\}$, $F \cap F_2=\{i_2\}$. Then 
\[u= \mathbf{x}_F \cdot \mathbf{x}_{F_1 \setminus \{i_1\}} \cdot \mathbf{x}_{F_2 \setminus \{i_2\}} \cdot \mathbf{x}_{F_3} \cdot \ldots \cdot \mathbf{x}_{F_t}. 
\]
By our assumption, there exists $1 \leq k\leq 3$, such that $\mathbf{x}_{(F_1 \setminus \{i_1\}) \cup \{i_k\}} \in J$. If $k \in \{1,2\}$, then 
\[u=\mathbf{x}_{(F_1 \setminus \{i_1\}) \cup \{i_k\}} \cdot \mathbf{x}_{(F_2 \setminus \{i_2\}) \cup \{i_3\}} \cdot \mathbf{x}_{F_3} \cdot \ldots \cdot \mathbf{x}_{F_t} \cdot \mathbf{x}_{F \setminus \{i_3,i_k\}}  \in J(\fm^{[3])^{t-1}}.
\]
Otherwise, $k=3$ and $u=\mathbf{x}_{(F_1 \setminus \{i_1\}) \cup \{i_k\}} \cdot \mathbf{x}_{F_2} \cdot \ldots \cdot \mathbf{x}_{F_t} \cdot \mathbf{x}_{F \setminus \{i_2,i_k\}} \in J(\fm^{[3])^{t-1}}$.
\end{itemize}

\textbf{Case 2.} $t=2$ and $|F \cap (F_1 \cup F_2)|=3$. 
\begin{itemize}
\item[] Again in this case, without loss of generality, we may assume that $F \cap F_1=\{i_1,i_2\}$ and $i_3 \in F_2$. Then $u=\mathbf{x}_F \cdot \mathbf{x}_{F_1 \setminus \{i_1,i_2\}} \cdot \mathbf{x}_{F_2 \setminus \{i_3\}}$. By our assumption, there exists $1 \leq k \leq 3$ such that $\mathbf{x}_{(F_3 \setminus \{i_3\}) \cup \{i_k\}} \in J$. Consequently,
\[u=\mathbf{x}_{(F_3 \setminus \{i_3\}) \cup \{i_k\}} \cdot \mathbf{x}_{(F_1 \setminus \{i_1, i_2\}) \cup F\setminus \{i_k\}} \in J\fm^{[3]}.\]
\end{itemize}

\textbf{Case 3.} $t=3$ and $|F \cap (F_1 \cup F_2 \cup F_3)|=3$.
\begin{itemize}
\item[] Using the assumption on $F, F_1, F_2, F_3$ and by skipping the symmetric cases, there are two possibilities to consider:
\begin{itemize}
\item[(1)] $F \cap F_1=\{i_1,i_2\}$ and $i_3 \in F_2$.\\
In this case, an argument similar to case (2) yields the required result.
\item[(2)] $F \cap F_k =\{i_k\}$, for $k=1,2,3$.\\
In this case, an argument similar to case (1) yields the required result.
\end{itemize}
\end{itemize}
\end{proof}

\section{The Valabrega-Valla module of an ideal}\label{c3}
Let $R = \mathbb{K}[x_1, \ldots, x_n]$ be the polynomial ring over a field $\mathbb{K}$, $J \subseteq R$ be a homogeneous ideal and $I \subseteq R$ be the Jacobian ideal of $J$, by which we always mean the ideal $\left( J , I_r(\Theta)\right)$ where $r = \mathrm{ht}(J)$
and $\Theta$ stands for the Jacobian matrix of a set of generators of $J$. More precisely, if $J = ( f_1, \ldots, f_s)$, then
\[
\Theta= \begin{bmatrix}
\frac{\partial f_1}{\partial x_1} & \frac{\partial f_1}{\partial x_2} &\cdots & \frac{\partial f_1}{\partial x_n}\\
\vdots & \vdots & & \vdots \\
\frac{\partial f_s}{\partial x_1} & \frac{\partial f_s}{\partial x_1} & \cdots & \frac{\partial f_s}{\partial x_n}
\end{bmatrix}.
\]
Note that $I_r(\Theta)$ is independent from the choice of generators of $J$. In the following, we consider the pair $J \subseteq \left( J , I_r(\Theta)\right)$ and we simply write $\VaVa_J$  instead of $\VaVa_{J \subseteq \left( J , I_r(\Theta)\right)}$.

\begin{ex}\mbox{}\label{Jacobmax}
\begin{itemize}
\item[(a)] Let $\mathbb{X}$ be a finite set of $r$ points in the projective space $\mathbb{P}_{\mathbb{K}}^{n-1}$ over an algebraically closed field $\mathbb{K}$. Denote by $J$ the defining ideal of $\mathbb{X}$. If $\VaVa_{J}\neq \{0\}$, then $\indeg(\VaVa_{J}) \leq r$~\cite[Proposition 1.3]{AZ}.  
\item[(b)] Let  $\fm^d$ be the $d^{\mathrm{th}}$-power of irrelevant  maximal ideal of $R$. The Jacobian matrix is of the form 
\[
\Theta(\fm^d)=\left[\begin{array}{l|ccccc}
dx_1^{d-1}&&&&&\\
(d-1)x_1^{d-2}x_2&&&&&\\
\vdots&&&\textbf{*}&&\\
x_n^{d-1}&&&&&\\
\hline
0&&&&&\\
\vdots&&&\Theta'&&\\
0&&&&&\\
\end{array}
\right]
,\]
where $\Theta'$ is the Jacobian matrix of ideal $(x_2,\ldots,x_n)^{d}$. We use induction on $n$, to show that
$I_{n}(\Theta)=\fm^{n(d-1)}$. Our induction hypothesis implies that, $I_{n-1}(\Theta')=(x_2,\ldots,x_n)^{(n-1)(d-1)}$. Changing the role of $x_1$ by $x_i$, we obtain  
$$\fm^{d-1}(x_1,\ldots,\hat{x}_i,\ldots,x_n)^{(n-1)(d-1)}\subseteq I_n(\Theta),$$ for $i=1,\ldots,n$. Hence 
\begin{align}\label{Abi}
I_n(\Theta)\supseteq \fm^{d-1}(\sum_{i=1}^n (x_1,\ldots,\hat{x_i},\ldots,x_n)^{(n-1)(d-1)}.
\end{align}
We claim that the latter ideal is equal to $\fm^{n(d-1)}$. Let $x_1^{\alpha_1}\cdots x_n^{\alpha_n}\in \mathcal{G}(\fm^{n(d-1)})$. Then $\sum \alpha_i=n(d-1)$. Hence there exists $j$ such that $\alpha_j\leq d-1$. For $k\neq j$ we choose $0 \leq \beta_k\leq \alpha_k$ such that $\alpha_j+\sum_{k\neq j}\beta_k=d-1$. Then we have 
\[x_1^{\alpha_1}\cdots x_n^{\alpha_n}=(x_j^{\alpha_j}x_1^{\beta_1}\cdots \hat{x}_j\cdots x_n^{\beta_n})\cdot(x_1^{\alpha_1-\beta_1}\cdots \hat{x}_j\cdots x_n^{\alpha_n-\beta_n}). \]
Note that $\sum_{k\neq j}(\alpha_k-\beta_k)=(n-1)(d-1)$. Hence $x_1^{\alpha_1}\cdots x_n^{\alpha_n}$ belongs to the right side of~\eqref{Abi}. Now from \cite[Example 2.19]{AA}, we conclude that  $\VaVa_{\mathfrak{m}^{d}} = \{0\}$.
%\item[(c)]$N\VaVa_{\mathfrak{m}^{[d]}} = 1$.
\end{itemize}
\end{ex}
Let $J=I(\mathcal{C})$ be the facet ideal of the $d$-uniform clutter $\mathcal{C}$. In the sequel,  we give a combinatorial criterion for which $\VaVa_{J}\neq \{0\}$ (Theorem~\ref{ttorsion}). As a consequence we recover~\cite[Theorem 3.3]{AR}.

\begin{defn}
Let $\C$ be a $d$-uniform clutter on $[n]$. A $(d-1)$-subset $e \subset [n]$ is called an \textit{submaximal circuit} of $\C$, if there exists $F \in \C$, such that $e \subset F$. The set of all submaximal circuits of $\C$ is denoted by ${\rm SC}(\C)$. For $e \in \mathrm{SC} \left( \C \right)$, the \textit{neighbourhood} of $e$, $N \left( e \right)$, is defined as follows:
\[
N \left( e \right) = \left\{ v \in [n] \colon \quad \{v\} \cup e \in \C \right\}.
\]
Also, for $e_1, \ldots, e_s \in \mathrm{SC} \left( \C \right)$, let $N \left( e_1, \ldots, e_s \right)= \cup_{i=1}^{s} N \left( e_i \right)$.
\end{defn}

\begin{defn}
Let $\C$ be a $d$-uniform clutter with vertex set $[n]$. A subset $ A \subset [n]$ is called \textit{independent}, if there is no circuit in $\C$ which is contained in $A$.
\end{defn}

Let $\mathcal{C}$ be a $d$-uniform clutter on vertex set $[n]$. For a subset $A \subset [n]$, let ${A \choose {d-1}}$ denotes the set of all $(d-1)$-subsets of $A$. Then we define:
\begin{align*}
\alpha(A) := \max \bigg\{ r \colon \quad \text{ there exist } & e_1, \ldots, e_s \in {\rm SC}\left( \mathcal{C} \right) \cap {A \choose {d-1}} \\
& \text{ such that } |N \left( e_1, \ldots, e_s \right)| = r \bigg\}.
\end{align*}

%\medskip
Let $M$ be an $m \times n$ matrix and $1 \leq r \leq \min\{m,n\}$ be an integer. A \textit{transversal} of length $r$ in $M$ or an \textit{$r$-transversal} of $M$ is a collection of $r$ entries of $M$ with different rows and columns. In other words, an $r$-transversal of $M$ is the entries of the main diagonal of an $r \times r$ sub-matrix of $M$ after suitable changes of columns and rows.

\begin{lem} \label{g in I so alpha is greater than height}
Let $\mathcal{C}$ be a $d$-uniform clutter on vertex set $[n]$ and $J = I\left( \mathcal{C} \right)$ its facet ideal. Let $r$ be a positive integer and $\Theta$ denotes the Jacobialn matrix of $J$. If $x_{i_1}^{\beta_1} \cdots x_{i_m}^{\beta_m} \in I_r \left( \Theta \right)$, then $\alpha \left( \left\{ {i_1}, \ldots, {i_m} \right\} \right) \geq r$.
\end{lem}

\begin{proof}
Let $A_{p \times p}$ be a square submatrix of $\Theta$. We can see from the proof of \cite[Lemma 3.1]{AR} that $\mathrm{det}(A)=\beta u_1 \cdots u_p$, where $u_1, \ldots, u_p$ is a $p$-transversal in $A$ and $\beta \in \mathbb{K}$. In particular, 
$$\left( \mathrm{det}(A) \colon \; A_{r \times r} \text{ is a square submatrix of } \Theta \right),$$
leads to a monomial generator for $I_r(\Theta)$ up to cancellation of scalar coefficient.

If $x_{i_1}^{\beta_1} \cdots x_{i_m}^{\beta_m} \in I_r \left( \Theta \right)$, then there exists a square submatrix $A_{r \times r}$ of $\Theta$ such that $\mathrm{det}(A) \mid x_{i_1}^{\beta_1} \cdots x_{i_m}^{\beta_m}$. By the above discussion, $\mathrm{det}(A)=\beta u_1 \cdots u_r$ where $u_1, \ldots, u_r$ is an $r$-transversal in $A$ and $\beta\in \mathbb{K}$. Thus each $u_j$ is of the form $\mathbf{x}_{e_j}$, where
\[
e_{j} \in \mathrm{SC} \left( \mathcal{C} \right ) \cap \binom{\left\{ {i_1}, \ldots, {i_m} \right\}}{d-1},
\]
for $j=1,\ldots, r$. Now, it is obvious that $N \left( e_{1}, \ldots, e_{r} \right) \geq r$. This completes the proof.
\end{proof}

%Let us to call an ideal $J\subset R$ \textit{Aluffi $t$-torsion}, if $J \cap I^t \neq JI^{t-1}$, where $I$ is the Jacobian  ideal of $J$

\begin{thm}\label{ttorsion}
Let $\mathcal{C}$ be a $d$-uniform clutter on vertex set $[n]$ and $J = I \left( \mathcal{C} \right)$ its facet ideal. Let $1 <t \leq d$ and $r$ be positive integers and $y_1, \ldots y_m$, ($m \geq d$), be distinct vertices of $\mathcal{C}$, such that:
\begin{itemize}
\item[\rm (i)] $F := \left\{ y_1, \ldots, y_d \right\} \in \mathcal{C}$;
\item[\rm (ii)] For any $(t-1)$-subset $B$ of $\{y_1 \ldots, y_t\}$, the set $B \cup \{y_{t+1}, \ldots, y_m \}$ is independent;
\item[\rm (iii)] $\alpha \left( y_{t+1}, \ldots, y_m\right) = r -1$.
\end{itemize}
Then $J \cap \left( J, I_r(\Theta) \right)^t \neq J \cdot \left( J, I_r(\Theta) \right)^{t-1}$. In particular under the above conditions, $\left( \VaVa_J \right)_t \neq \{0\}$.
\end{thm}

\begin{proof}
Let $A := \{y_{t+1}, \ldots, y_m\}$. Since $\alpha (A) = r-1$, there exist $e_1, \ldots, e_s \in {\rm SC}(\mathcal{C}) \cap {A \choose {d-1}}$, such that $|N \left( e_1, \ldots, e_s \right)|=r-1$. Put
\begin{align*}
& A_1 := N \left( e_1 \right),\\
& A_i := N \left( e_i \right) \setminus N \left( e_1, \ldots, e_{i-1} \right), \text{ for } i>1.
\end{align*}
and $N_i := |A_i|$. Then, $A_i \cap A_j = \varnothing$ and without loss of generality, we may assume that $A_i \neq \varnothing$. This implies that, the elements
\begin{equation*}
\underbrace{{\rm \textbf{x}}_{e_1}, \ldots, {\rm \textbf{x}}_{e_1}}_{N_1 \text{ times}}, \; \ldots \; , \underbrace{{\rm \textbf{x}}_{e_s}, \ldots, {\rm \textbf{x}}_{e_s}}_{N_s \text{ times}}
\end{equation*}
form a $(r-1)$-transversal in $\Theta$. 

Now, for $i=1, \ldots, t$, take the monomials $g_i \in S$, as follows:
$$g_i = \frac{{\rm \textbf{x}}_F}{x_{y_i}} \cdot {\rm \textbf{x}}_{e_1}^{N_1} \cdots {\rm \textbf{x}}_{e_s}^{N_s}.$$
We claim that $g:= \prod_{i=1}^t g_i \in J \cap \left( J, \, I_r\left( \Theta \right) \right)^t \setminus J  \left(J, \, I_r\left( \Theta \right) \right)^{t-1}$.

By (i) it is clear that $g \in J$. Also, (ii) implies that, $N \left(e_1, \ldots, e_s \right) \subset [n] \setminus \{y_1, \ldots, y_t\}$. In particular, the elements
\begin{equation*}
\frac{{\rm \textbf{x}}_F}{x_{y_i}}, \; \underbrace{{\rm \textbf{x}}_{e_1}, \ldots, {\rm \textbf{x}}_{e_1}}_{N_1 \text{ times}}, \; \cdots, \; \underbrace{{\rm \textbf{x}}_{e_s}, \ldots, {\rm \textbf{x}}_{e_s}}_{N_s \text{ times}},
\end{equation*}
form an $r$-transversal in $\Theta$. Hence $g_i \in  I_r\left( \Theta \right)$, for $i= 1, \ldots t$. It remains to show that $g \notin J  \cdot \left( J , \, I_r\left( \Theta \right) \right)^{t-1}$.

To show this, first note that
$$J \cdot \left(J , \, I_r\left( \Theta \right) \right)^{t-1} = J^t + J^{t-1} \cdot I_r\left( \Theta \right) + \cdots + J \cdot I_r\left( \Theta \right)^{t-1}.$$
Being a monomial ideal, it suffices to show that $g \notin J^{t-j} \cdot I_r\left( \Theta \right)^{j}$, for $j=0, \ldots, t-1$. 

Let us show that  $g \notin J^t$. Otherwise, there exist $F_1, \ldots, F_t \in \C$, such that $ \textbf{x}_{F_1} \cdots \textbf{x}_{F_t} \mid g$. In particular $F_i \subseteq \mathrm{supp} (g) \subseteq \{ y_1, \ldots, y_m \}$, for $i=1, \ldots, t$. In this case, (ii) implies that $F_i \supseteq \{ y_1, \ldots, y_t \}$ which means that $x_{y_1}^t \cdots x_{y_t}^t \mid g$. This is impossible by our choice of $g$.

However, if $j \geq 1$ and $g \in J^{t-j} \cdot I_r\left( \Theta \right)^{j}$, then there exists $g' \in \mathcal{G} \left( I_r \left(\Theta \right) \right)$ such that $g' \mid g$ but $x_{y_1} \cdots x_{y_t} \nmid g'$. It follows from (ii) that $\mathrm{supp} (g') \subseteq A$. But lemma~\ref{g in I so alpha is greater than height} implies that $\alpha \left( \mathrm{supp} (g') \right) \geq r$ which contradicts to (iii). This completes the proof.

\end{proof}

As a direct consequence of Theorems~\ref{ttorsion}, we may recover one direction of \cite[Theorem 3.3]{AR}. For being self contained, we write a slightly shorter proof for other direction of \cite[Theorem 3.3]{AR} as well.

\begin{prop}[{\cite[Theorem 3.3]{AR}}]
Let $G$ be a graph and $\mathrm{ht}(I(G)) = r > 1$. Then the followings are equivalent:
\begin{itemize}
\item[\rm (a)] $\indeg(\VaVa_{I(G)})=2$
\item[\rm (b)] there are adjacent vertices $x_1$, $x_2$ and other vertices $x_{i_1},\ldots, x_{i_s}$, for some integer $s \geq 1$, such that
\begin{itemize}
\item [\rm (1)] both of the sets $\{x_1, x_{i_1},\ldots, x_{i_s} \}$ and $\{x_2, x_{i_1},\ldots, x_{i_s} \}$ are independent in $G$.
\item[\rm (2)] $|N(\{x_{i_1},\ldots, x_{i_s}\})|=r-1$.
\end{itemize}
\end{itemize}
\end{prop}

\begin{proof}
Let $J=I(G)$ be the edge ideal of $G$. We note that the ideal $I_r(\Theta)$ is a monomial ideal, where $\Theta$ is the Jacobian matrix of $J$ \cite[Lemma 3.1]{AR}.

(a) $\to$ (b): Since $\indeg(\VaVa_{I(G)})=2$, it follows that  $J \cap \left( J, I_r(\Theta) \right)^2 \nsubseteq J \cdot \left( J, I_r(\Theta) \right)$. Pick a monomial $g \in J \cap \left( J, I_r(\Theta) \right)^2 \setminus J \cdot \left( J, I_r(\Theta) \right)$. Then $g=g_1g_2$ where $g_i$ is a monomial in $\left( J, I_r(\Theta) \right)$. If $g_i \in J$ for some $i=1,2$ then $g=g_1g_2 \in J \cdot \left( J, I_r(\Theta) \right)$, which is a contradiction. Hence $g_i \in I_r(\Theta) \setminus J$. However, $g \in J$ which implies that there are adjacent vertices $x_1,x_2$ in $G$ such that $x_1x_2 \mid g$. Since $g_i \notin J$, we conclude, without loss of generality, that $x_i \mid g_i$, for $i=1,2$. Write $g_1=x_1 x_{i_1}^{\alpha_{i_1}} \cdots x_{i_s}^{\alpha_{i_s}}$ and $g_2=x_2 x_{j_1}^{\beta_{j_1}} \cdots x_{j_t}^{\beta_{j_t}}$, where $\sum_k \alpha_{i_k}=\sum_k \beta_{j_k} =r-1$. Then the sets $A=\{x_1, x_{i_1}, \ldots, x_{i_s} \}$ and $B=\{x_2, x_{j_1}, \ldots, x_{j_t} \}$ are independent, because $g_i \notin J$, for $i=1,2$. If $x_1$ is adjacent to some vertex in $B \setminus \{x_2\}$ and simultaneously $x_2$ is adjacent to some vertex in $A \setminus \{x_1\}$, then $g \in J \cdot \left( J, I_r(\Theta) \right)$ which is a contradiction. Assume that $x_2$ is not adjacent to any vertex in $A \setminus \{x_1\}$. Then clearly the adjacent vertices $x_1, x_2$ together with $\{ x_{i_1}, \ldots, x_{i_s} \}$ satisfy in (1). Lemma~\ref{g in I so alpha is greater than height} implies that $N(\{x_{i_1}, \ldots,x_{i_s}\}) \geq r-1$, for $x_{i_1}^{\alpha_{i_1}} \cdots x_{i_s}^{\alpha_{i_s}} \in I_{r-1}(\Theta)$ by \cite[Lemma 3.2]{AR}. On the other hand, $x_{i_1}^{\alpha_{i_1}} \cdots x_{i_s}^{\alpha_{i_s}} \cdot x_{j_1}^{\beta_{j_1}} \cdots x_{j_t}^{\beta_{j_t}} \notin I_r(\Theta)$, this means that for any subset $C$ of $\{x_{i_1}, \ldots,x_{i_s}, x_{j_1}, \ldots,x_{j_t}\}$, $|N(C)| < r$ (c.f. \cite[Lemma 3.2]{AR}). Thus $N(\{x_{i_1}, \ldots,x_{i_s}\})=r-1$, as required.

(b) $\to$ (a): This implication follows from Theorem~\ref{ttorsion} in special case $d=2$.
\end{proof}

\section{application: the rees algebra of $I/J$ when $\VaVa_{J\subseteq I}=\{0\}$}\label{Equ.Rees}
Let $J\subseteq I\subseteq R$ be ideals in a Noetherian ring $R$. We have seen that $\VaVa_{J\subseteq I}=\{0\}$ if and only if  the Aluffi algebra of $I/J$ is isomorphic with the corresponding Rees algebra.
By~\cite[Lemma 1.2]{AA}, the Aluffi algebra has the following presentation:
\begin{equation}\label{pre.Aluffi}
\mathcal{A}_{R/J}(I/J)\simeq \dfrac{\mathcal{R}_R(I)}{(J,\tilde{J})\mathcal{R}_R(I)},
\end{equation}
where $J$ is in degree zero and $\tilde{J}$ is in degree one in $\mathcal{R}_R(I)$.  
 Then to describe the defining ideal of the Rees algebra of $I/J$, we need just to find the defining ideal of the Rees algebra of $I$. 
 In this section, we find explicit equation for the defining ideal of the Rees algebra of $I/J$ when  $I$ is a monomial ideal in the ring $R=\mathbb{K}[\mathbf{x}]$ and $\VaVa_{J\subseteq I}=\{0\}$.

Let $I$ be a monomial ideal in a polynomial ring $R$ and $\mathcal{G}(I)=\{f_1,\ldots,f_m\}$. Denote by $I_s$ the set of all non-decreasing sequences of integers $\alpha=(i_1,\ldots, i_s) \subseteq   \{1,\ldots,m\}$. Then $f_{\alpha}=f_{i_1}\ldots f_{i_s}$ is the corresponding product of monomials in $I$. Let $T_{\alpha}=T_{i_1}\cdots T_{i_s}$ be the corresponding product of $T_i$ in $S=R[T_1,\ldots, T_m]$.  For every $\alpha,\beta\in I_s$ we consider the binomial 
$$T_{\alpha,\beta}=\frac{f_{\beta}}{\gcd(f_{\alpha},f_{\beta})}T_{\alpha} -\frac{f_{\alpha}}{\gcd(f_{\alpha},f_{\beta})}T_{\beta}.$$
By a result  in \cite{taylor} on the defining ideal of the Rees algebra of a monomial ideal, we have 
 \[{\mathcal R}_R(I)\simeq \frac{R[T_1,\cdots,T_m]}{(I_1({\mathbf T}.\phi),  \bigcup_{s=2}^{\infty} P_s)},\] 
where  $ I_1({\mathbf T}\cdot\phi)$ is the ideal generated by one minors of the product of matrix $\mathbf{T}=[T_1\ T_2\ \ldots \ \ T_m]$  by the first syzygy matrix $\phi$ of $I$ and $P_s=(\{T_{\alpha,\beta}\colon\ \alpha,\beta\in I_s \})$. Note that  $ I_1({\mathbf T}\cdot\phi)$ is the defining ideal of the symmetric algebra of $I$. Thus by~\eqref{pre.Aluffi}, we obtain the following presentation:
 \[{\mathcal A}_{R/J}(I/J)\simeq \frac{R[T_1,\ldots,T_n]}{\left (J,{\tilde J}, I_1({\mathbf T}\cdot\phi),  \bigcup_{s=2}^{\infty} P_s\right )}.\]  
\begin{ex}
Let $R=\mathbb{K}[x_1,\ldots, x_n]$ and  $I=\fm^d$ be the $d^{\mathrm{th}}$ power of the irrelevant maximal ideal of $R$ ordered by  lexicographic order with $x_1>x_2>\cdots> x_n$. Let 
$$\phi\colon R[T_1,\ldots,T_N] \surjects \mathcal{R}_R(I)$$ be the $R$-algebra homomorphism taking $T_i$ to the $i$th monomial of degree $d$ in $x_1,\ldots,x_n$ in lexicographic order where $N=\binom{d+n-1}{n-1}$. The kernel of $\phi$ is the defining ideal $\mathcal{J}$ of the Rees Algebra of $\fm^d$. Write $m_1,\ldots, m_r$ for the monomials of degree $d-1$ in $x_1,\ldots,x_n$  in lexicographic order where $r=\binom{d+n-2}{n-1}$. Let $\mathbf{M}$ be a matrix of size $n\times r$ whose $(i,j)$th entry is the variable $T_k$ such that $\phi(T_k)=x_im_j$. Let $\mathbf{X}$ be the variable matrix of size $n\times 1$.  Denote by $\mathcal{Q}=[\mathbf{X}\ | \ \mathbf{M}]$ the  concatenation of $\mathbf{X}$ and  $\mathbf{M}$. By \cite[Theorem 4]{Barshay}, we have $\mathcal{J}=I_2(\mathcal{Q})$, the $2\times 2$ minors of $\mathcal{Q}$. Note that generators of $I_2(\mathcal{Q})$ involving  the  variable column $\mathbf{X}$ is the defining ideal of the symmetric algebra of $I$.  

Now let $J\subseteq \fm^d$ be an ideal generated by some $d$-forms in $R$.  By Proposition~\ref{dpower}(b), $\VaVa_{J\subseteq \fm^d}=\{0\}$ and 
\[\mathcal{R}_{R/J}(\fm^d/J)\simeq R[T_1,\ldots,T_N]/(J,\tilde{J},I_2(\mathcal{Q})). \]
\end{ex}

Let $R$ be a standard graded ring with irrelevant maximal ideal $\fm$  and $I\subset R$ an ideal, the \textit{special fiber} of $I$ is defined to be $\mathcal{F}(I)=\mathrm{gr}_I(R)\otimes R/\mathfrak{m}$, where $\mathrm{gr}_I(R)=\mathcal{R}_R(I)/I\mathcal{R}_R(I)=\bigoplus_{i\geq 0}I^i/I^{i+1}$. In the case that $R=\mathbb{K}[x_1, \ldots, x_n]$  and $I=(f_1,\ldots,f_m)$ a homogeneous ideal, the special fiber $\mathcal{F}(I)$ is isomorphic to $\mathbb{K}[f_1,\ldots, f_m]$. Then there is a homomorphism $\varPsi \colon \mathbb{K}[T_1,\ldots,T_m] \surjects \mathcal{F}(I)$ that maps
$T_i$ to $f_i$. Set $\mathcal{H}=\ker \varPsi$. The ideal $I$ is called of \textit{fiber type} if $\mathcal{J}=S\mathcal{J}_1+S\mathcal{H}$, where $\mathcal{J}_1$ is the degree one homogeneous part of the defining ideal of the Rees algebra of $I$  and $S=R[T_1, \ldots, T_m]= \mathbb{K}[x_1, \ldots, x_n, T_1, \ldots, T_m]$.

Let $G$ be a simple graph on the vertex set $[n]$  and $I(G)$ the edge ideal of $G$. Let $w=\{v_0,v_1,\ldots,v_r=v_0\}$ be an even closed walk in $G$ and $
f_i=x_{v_{i-1}}x_{v_i}$. Since $f_1f_3\cdots f_{r-1}=f_2f_4\cdots f_{r}$, it follows that the binomial $T_{w}=T_1T_3\cdots T_{r-1}-T_2T_4\cdots T_{r}$ belongs to the defining ideal of the $\mathbb{K}$-algebra $\mathbb{K}[I(G)]$. 
Set 
\begin{align*}
& P(G)=\left( \{T_w\colon \ w\  \mbox{is an even closed walk in } G\} \right), \text{ and }\\ 
& P'(G)=\left( \{T_w\colon w\ \mbox{is an even cycle in}\ G\} \right).
\end{align*}
\begin{prop}\label{Pres. Rees}
Let $J\subseteq I$ be ideals in the ring $R=\mathbb{K}[\mathbf{x}]$ such that $I=I(G)$ is the edge ideal of a simple graph $G$  and $\VaVa_{J\subseteq I}=\{0\}$. Then 
\[\mathcal{R}_{R/J}(I/J)\simeq \frac{R[T_e\colon e\in E(G) ]}{(J,\tilde{J},\mathcal{J}_1,P(G))}.\]
Moreover, if $G$ is a bipartite graph, then
\[\mathcal{R}_{R/J}(I/J)\simeq \frac{R[T_e\colon e\in E(G) ]}{(J,\tilde{J},\mathcal{J}_1,P'(G))}.\] 
\end{prop}
\begin{proof}
By~\cite[Theorem 3.1]{RV1}, the ideal $I(G)$ is of fiber type
and $$\mathcal{R}_R(I)\simeq R[T_e\colon e\in E(G)]/(\mathcal{J}_1,P(G)),$$ 
where $\mathcal{J}_1$ is the defining ideal of the symmetric algebra of $I(G)$. Moreover, in the case that $G$ is a bipartite graph $$\mathcal{R}_R(I)\simeq R[T_e\colon e\in E(G)]/(\mathcal{J}_1,P'(G)). $$ 
Therefore, by~\eqref{pre.Aluffi} and the fact that the Aluffi algebra is isomorphic with the Rees algebra, we get the required presentations. 
\end{proof}
\begin{ex}
Let $I=I(C_6)+J$, where $J=(x_7x_9,x_8x_9)$. By  Proposition~\ref{graph vv}, $\VaVa_{J\subseteq I}=\{0\}$. The defining ideal of the Rees algebra of $I/J$ contains the form $T_1T_3T_5-T_2T_4T_6$  corresponding to the cycle $C_6$ as a minimal generator by  Proposition~\ref{Pres. Rees}. Hence $\mathrm{rt}(I/J)=3$ while $t_0(J)=2$.
\end{ex}

\begin{nota}
Let $\C$ to be a complete $d$-partite $d$-uniform clutter with the $d$-partition $\{V_i \colon \; i\in [d]\}$ and $e\in \C$. Consider the
ring homomorphism
$$\phi_{e}\colon S= R[\{T_{e'}\colon \; e\neq e'\in \C\}]\surjects S_{e}$$ that sends  $T_{e'}$ to $\frac{\mathbf{x}_{e'}}{\mathbf{x}_{e}}$. Set
$J_{e}=\ker \phi_{e}$. Moreover, for $e\neq e'\in \C$ we fix a vertex $v(e,e')\in e\setminus e'$ such that $v(e,e')$ and $v(e',e)$ lie in the same partition and by $v_e(j)$ we mean the only vertex of $e$ in the same partition as the vertex $j$. Finally we denote the circuit of $\C$ obtained from $e$ by replacing $j$ instead of $v_e(j)$ by $e(j)$.
\end{nota}

Assume that $j= v(e,e')\in V_i$. Then since $j'=v(e',e)$ is in $V_i\cap e'$ we have $j'=v_{e'}(j)$ and similarly $j=v_e(j')$. In this case, $e(j')$ and $e'(j)$ are the circuits obtained from $e$ and $e'$ respectively, by ``\textit{swapping}'' those vertices of $e$ and $e'$ which lie in the $i$'th partition. For example, $e(j')= (e\cup\{j'\}) \setminus \{v_e(j')\}= (e\cup\{j'\}) \setminus \{j\}$.
\begin{prop}
Let $\mathcal{C}$ be complete $d$-partite $d$-uniform clutter on vertex set $[n]$ and $\mathcal{C}' \subseteq \mathcal{C}$. Then 
\[\mathcal{R}_{R/I(\mathcal{C}')}(I(\mathcal{C})/I(\mathcal{C}'))\simeq  \frac{R \left[ T_e\colon \;\ e\in \C \right]}{(I(\mathcal{C}'), \mathcal{A})+(T_e \colon \; e \in \mathcal{C}')}, \]
where $\mathcal{A}$ is generated by the set of all binomials of the form $T_ex_i- x_rT_{e(i)}$ with $e\in \C$, $i\in [n]\setminus e$ and $r=v_{e}(i)$ together with those of the form $T_eT_{e'}-T_{e(j')}T_{e'(j)}$ where $j= v(e,e')$ and $j'=v(e',e)$, for $e\neq e'\in \C$ with $|e'\setminus e|>1$. In particular, the relation type number of $I(\mathcal{C})/I(\mathcal{C}')$ is at most $2$. 
\end{prop}
\begin{proof}
The assertion follows from Proposition~\ref{d-uni-d-part} and~\cite[Theorem 4.2]{AAAR}. 
\end{proof}
The above proposition can be applied to produce examples of a pair $J\subset I$ such that $t_0(J)>\mathrm{rt}(I/J)$. 
We close this paper by posing the following questions. 
\begin{question}
Find (if it is possible) a class of ideals $J\subseteq I$ such that $t_0(J)<\mathrm{AR}(J,I)$. 
\end{question}
\begin{question}
Is Theorem~\ref{vvmon} valid for the case that $J\subseteq I$ are homogeneous ideals in $\mathbb{K}[\mathbf{x}]$?
\end{question}

\begin{question}
Let $J$ be a square-free monomial ideal generated in degree $d$. Characterize when $\VaVa_{J\subseteq \fm^{[d]}}=\{0\}$? (c.f. Proposition~\ref{square-free pairs}). 
\end{question}
\begin{question}
Is the converse of Theorem~\ref{ttorsion} true? 
\end{question}

\section*{Acknowledgment}
The authors would like to thank Rashid Zaare-Nahandi for fruitful  discussion and comments on the subject of the paper. The authors would like to thank M. Farrokhi D.G. for proposing Example~\ref{biparti-subgrpah}, A. Taherkhani for suggesting the names almost $C_3$-embedded and almost $P_3$-embedded subgraphs and F. Planas-Vilanova for pointing out the inequality $\mathrm{AR}(J,I)\leq \mathrm{rt}(I/J)$.

\end{document}